\documentclass{article}

\advance\textheight -\headheight
\advance\textheight -\headsep
\oddsidemargin\paperwidth
\advance\oddsidemargin -\textwidth
\divide\oddsidemargin2 
\ifdim\oddsidemargin<.5truein \oddsidemargin.5truein \fi
\advance\oddsidemargin -1truein
\evensidemargin\oddsidemargin
\topmargin\paperheight \advance\topmargin -\textheight
\advance\topmargin -\headheight \advance\topmargin -\headsep
\divide\topmargin2
\ifdim\topmargin<.5truein \topmargin.5truein \fi
\advance\topmargin -1.4truein\relax
\usepackage{graphicx}
\graphicspath{{./images/}}
\DeclareGraphicsExtensions{.pdf,.png,.jpg,.jpeg}

\usepackage[utf8]{inputenc}
\usepackage[T2A]{fontenc}
\usepackage{amsmath,amssymb,amsthm}
\usepackage[a4paper,hmargin=2.5cm,vmargin=2.5cm]{geometry}
\usepackage[english]{babel}
\usepackage{amscd}

\usepackage{todonotes}
\usepackage{hyperref}
\usepackage{mathrsfs}
\usepackage{tikz-cd}
\usepackage{import}
\usepackage{placeins}
\usepackage{graphicx}
\usepackage[labelfont=bf]{caption}
\usepackage{subcaption}
\usepackage{enumitem}
\newtheorem{theorem}{Theorem}[section]

\newtheorem{prop}{Proposition}[section]
\newtheorem{corollary}{Corollary}[section]

\newtheorem{definition}{Definition}
\newtheorem{remark}{Remark}[section]
\newtheorem{example}{Example}[section]

  \usepackage[
    backend=biber,
    style=nature,
    sorting = nty
  ]{biblatex}

\title{Belyi Function Decompositions for The Icosahedron of Genus 4}
\author{Natalia Amburg \and Mariia Kovaleva}
\date{}

\begin{document}
\maketitle

\hfill Dedicated to the memory of A.V. Mikhalev.

\begin{abstract}
   The icosahedron $I_4$ of genus 4 is a dessin d'enfant embedded in Bring's curve $\mathcal{B}$. The dessin $I_4$ is related in some sense to a regular icosahedron $I_0$ embedded in the complex Riemann sphere. In particular, decompositions of Belyi functions $\beta_{I_0}: \mathbb{CP}^1 \rightarrow \mathbb{CP}^1$ and $\beta_{I_4}: \mathcal{B} \rightarrow \mathbb{CP}^1$ for $I_0$ and $I_4$ have the same lattice. The diagram of $\beta_{I_0}$ decompositions is already known. In the present paper we find $\beta_{I_4}$ decompositions. Note that $\beta_{I_0}$ decomposes into rational functions on $\mathbb{C}P^1$, while in case of $\beta_{I_4}$ we deal with maps between different algebraic curves.
   
   {\bf Key words}: dessins d'enfants, Belyi functions, functional composition, icosahedron of genus 4, Bring's curve.
\end{abstract}

\tableofcontents

\section{Introduction}
In the early 20th century Joseph Ritt proved his polynomial decomposition theorem. Suppose we have two decompositions of a polynomial $f$ into indecomposible polynomials of degrees greater than one
\begin{gather*}
    f = g_{1}\circ g_{2}\circ \cdots \circ g_{m} = h_{1}\circ h_{2}\circ \cdots \circ h_{n} \quad g_{i}\neq h_{i} \text{ for some i.} 
\end{gather*}
Then $m=n$, and the degrees of the components are the same but possibly in different order \cite{Ritt}.
However, the analogous problem for rational functions is more difficult. For example, a Belyi function $\beta_{I_0}: \mathbb{C}P^1\rightarrow\mathbb{C}P^1$ for a regular icosahedron $I_0$ has various decompositions with different number of indecomposable functions (see Figure \ref{im_i0}).
The notion of Belyi function is associated with dessins d'enfants (which means "child's drawing" in French). The theory of dessins d'enfants was developed by Alexander Grothendieck in 1980s, it connects topological and combinatorial objects with algebraic curves. A dessin d'enfant (or simply dessin) is a type of graph drawn on a topological surface with a single stroke of a pencil. A remarkable fact is that these drawings give us an exact characterisation of the algebraic curves.
In this paper we consider the icosahedron $I_4$ of genus 4, which has the automorphism group $A_5$ as well as a regular icosahedron $I_0$. Saying in a very brief way, the icosahedron $I_4$ of genus 4 is a dessin embedded in a genus-4 surfaces, on which $I_4$ induces the complex structure of Bring's curve $\mathcal{B}\subset \mathbb{CP}^4$ by its Belyi function $\beta_{I_4}:\mathcal{B} \rightarrow \mathbb{C}P^1$ (see Theorem \ref{existence_thr}). Decompositions of Belyi functions $\beta_{I_0}: \mathbb{CP}^1 \rightarrow \mathbb{CP}^1$ and $\beta_{I_4}: \mathcal{B} \rightarrow \mathbb{CP}^1$ have the same lattice. To decompose $\beta_{I_4}$ we consider quotients of $I_4$ by subgroups of its automorphism group. Also, we find equations of corresponding quotient curves. The full diagram of $\beta_{I_4}$ decompositions is presented in Figure \ref{full_gensch}.

\section{Dessin d’enfant}
\begin{definition}
[\textbf{Dessin d’enfant}] A dessin d’enfant is a bicolored graph embedded in a compact oriented two-dimensional manifold in such a way that:
\begin{itemize}
    \item The edges do not intersect.
    \item The complement of the graph is a disjoint union of regions homeomorphic to open disks.
    \item The vertices are colored in black and white in such a way that the adjacent vertices have opposite colors.
\end{itemize}
 The genus of a dessin is the genus of the underlying surface.   
\end{definition}
\begin{remark}
    In this paper we usually erase white vertices of degree 2 in order to simplify pictures.
\end{remark}
Dessins admit an encoding by permutations. Let $M$ be a dessin with $m$ edges. We label the edges from $1$ to $m$ and associate $M$ with the following pair of permutations $(\sigma, \alpha)$:
\begin{itemize}
    \item A cycle of $\sigma \in S_m$ contains the labels of edges incident with a black vertex, taken in counterclockwise direction around this vertex.
    \item The cycles of $\alpha \in S_m$ correspond, in the same way, to the white vertices.
\end{itemize}
\begin{remark}\cite{Lando}
    The permutation $(\sigma \alpha)^{-1}$ represents faces of the corresponding dessin. There are as many cycles in $(\sigma \alpha)^{-1}$ as there are faces. Also, the degree of a face is equal to the length of the corresponding cycle.
\end{remark}
\begin{example}
    \begin{figure}[hbt!]
        \centering
        \begin{normalsize} 
            \includegraphics[scale=0.3]{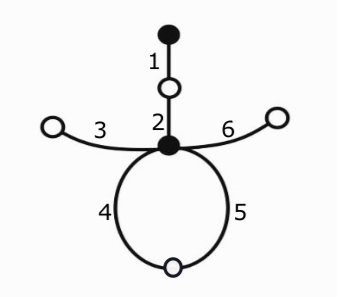}
       \end{normalsize}
       \caption{Dessin d’enfant}
       \label{im_perm}
    \end{figure}
    For the dessin of Figure \ref{im_perm} we obtain the following permutations:
    \begin{gather}
        \sigma = (1)(23456),\quad
        \alpha = (12)(3)(45)(6),\quad
        (\sigma \alpha)^{-1} = (12643)(5).   
    \end{gather}
\end{example}
\begin{prop} \label{triple_perm} \cite{Zvonkin}
    A pair of permutations $\sigma, \alpha \in S_m$ is associated with a dessin iff the permutation group generated by $\sigma, \alpha$ is transitive.
\end{prop}
\begin{definition}[\textbf{Belyi pair}] 
    A pair $(X,\beta)$, where $X$ is a Riemann surface, and $\beta: X\rightarrow \overline{\mathbb{C}}$ is a meromorphic function, is called a Belyi pair if all critical values of $\beta$ belong to $\{0,1,\infty\}$.
\end{definition}
\begin{theorem}[\textbf{Riemann's existence theorem}]\label{existence_thr}\cite{Zvonkin}
     For any dessin $M$ there exists a Belyi pair $(X, \beta)$ such that $M$ is embedded in $X$ as $M = \beta^{-1}([0,1])$ in such a way that:
     \begin{enumerate}
         \item All black vertices of $M$ are roots of the equation $\beta(x) = 0$, the multiplicity of each root being equal to the degree of the corresponding vertex.
         \item All white vertices of $M$ are roots of the equation $\beta(x) = 1$, the multiplicity of each root being equal to the degree of the corresponding vertex.
         \item Inside each face of $M$ there exists a single pole of $\beta$, the multiplicity of the pole being equal to the degree of the face.
     \end{enumerate}
     The pair $(X, \beta)$ is unique up to an automorphism of the surface $X$.
\end{theorem}
In other words, dessins are simple
topological encodings of these Belyi pairs. 
\begin{definition}[\textbf{Automorphisms of dessins}]
    Assume that a Belyi pair $(X, \beta)$ corresponds to a dessin $M$. Let as define an automorphism of the dessin $M$ as an automorphism $\varphi$ of the curve $X$ such that $\beta =\beta\circ\varphi$. 
\end{definition}
  

\section{Icosahedron of Genus 4}
Consider a regular icosahedron with white vertices in the middle of edges as a dessin embedded in the sphere and write down the corresponding permutations $(\sigma, \alpha)$.
Following to \cite{Zvonkin}, let us define \textbf{the icosahedron $I_4$ of genus 4} as a dessin associated with the pair of permutations $(\sigma^2, \alpha)$.
The definition is correct by Proposition \ref{triple_perm}. As you can check,
\begin{itemize}
    \item $\sigma^2$ has 12 cycles of length 5 $\Rightarrow$ $I_4$ has 12 (black) vertices of degree 5
    \item $\alpha$ has 30 cycles of length 2 $\Rightarrow$ $I_4$ has 30 edges (i.e. 30 white vertices of degree 2)
    \item $(\sigma^2\alpha)^{-1}$ has 12 cycles of length 5 $\Rightarrow$ $I_4$ has 12 faces of degree 5
\end{itemize}
 By Euler's formula, $I_4$ is embedded in the surface of genus 4. 
 \begin{figure}[h]
    \centering
    \includegraphics[scale=0.7]{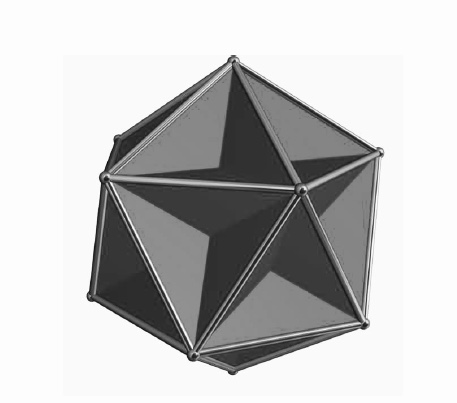}
    \caption{Great Dodecahedron}
    \label{im_i4_r3}
\end{figure}
 There is a nice way to visualize $I_4$ as the Great Dodecahedron (Figure \ref{im_i4_r3}), which is one of four nonconvex regular polyhedra. Note that vertices and edges of the Great Dodecahedron and a regular icosahedron are the same. 
 The permutation $(\sigma\alpha)^{-1}$ maps an edge incident with a vertex to the next edge in the cyclic order and we obtain triangular faces of a regular icosahedron. While the permutation $(\sigma^2\alpha)^{-1}$ maps an edge incident with a vertex to the one after the next and we obtain pentagonal faces of $I_4$. So, we can consider the Great Dodecahedron as an immersion of the surface of genus 4 into $\mathbb{R}^3$.
\subsection{Belyi function for $I_4$}
In 1786 Erland Bring developed a transformation to simplify a generic quintic equation to the form 
\begin{equation}\label{eq_p}
    p(x) = x^5 + ax + b.
\end{equation}
Note that the roots $x_i$ of the Bring quintic equation $p(x) = 0$ satisfy 
\begin{equation}\label{eqbring}
    \begin{cases}
        x_1 + x_2 + x_3 + x_4 + x_5 = 0\\
        x_1^2 + x_2^2 + x_3^2 + x_4^2 + x_5^2 = 0\\
        x_1^3 + x_2^3 + x_3^3 + x_4^3 + x_5^3 = 0
    \end{cases}
\end{equation}
In 1884 Felix Klein showed that homogeneous equations (\ref{eqbring}) cut out the curve $\mathcal{B}$ of genus 4 in $\mathbb{C}P^4$ and named it \textbf{Bring's curve}. Also, Klein proved that $I_4$ is embedded in $\mathcal{B}$ \cite{Klein}.
\begin{prop}\cite{Zvonkin}
    The automorphism group of Bring's curve $\mathcal{B}$ is the symmetric group $S_5$, which acts on $\mathcal{B}$ by permutations of five projective coordinates.
\end{prop}
\begin{example}
\begin{gather*}
    (12345) \in \text{Aut}(\mathcal{B}) \quad  (12345)(x_1:x_2:x_3:x_4:x_5) = (x_2:x_3:x_4:x_5:x_1)
\end{gather*}
\end{example}
\noindent A set of five roots of the equation $p(x) = 0$ usually gives rise to 120 points on Bring's curve $\mathcal{B}$. Indeed, the number of permutations of five roots $p(x)$ equals to $5! = 120$.
\begin{prop}\label{prop_beta_i4}
    A Belyi function $\mathcal{B}\rightarrow\mathbb{C}P^1$ for $I_4$ is given by
    \begin{equation}\label{eq_beta_i4}
        \beta_{I_4}(x)= 1 + \frac{25b^2}{128a^5}\cdot(125b^2 
        + \sqrt{5}\prod\limits_{1\leqslant i<j \leqslant 5}(x_i - x_j)),
    \end{equation}
    where $a = \sum\limits_{1 \leqslant i < j < k < s \leqslant 5}^5 x_i x_j x_k x_s, \quad b = - x_1x_2x_3x_4x_5$.
\end{prop}
To prove Proposition \ref{prop_beta_i4} we use Proposition \ref{prop_i4_i4*} and Remark \ref{prop_map_with_dual} discussed in \cite{Zvonkin}.

\noindent\textit{Notation.} Let $M$ be a dessin. By $M^*$ we denote a dessin dual to $M$, and $M \cup M^*$ stands for the original dessin $M$ and its dual $M^*$ on the same picture.
\begin{prop}\label{prop_i4_i4*}\cite{Zvonkin}
    A Belyi function for $H = I_4 \cup I_4^*$ is
     $$\beta_H(x_1:x_2:x_3:x_4:x_5) = \frac{256a^5}{256a^5 + 3125b^4},$$
    where $a = \sum\limits_{1 \leqslant i < j < k < s \leqslant 5}^5 x_i x_j x_k x_s, \quad b = - x_1x_2x_3x_4x_5$.
\end{prop}
\begin{remark}\label{prop_map_with_dual}\cite{Zvonkin}
    Let $\beta$ be a Belyi function for a dessin $M$, then the Belyi function for $M \cup M^*$ is $g \circ \beta$, where $g(y) = \frac{4y}{(y+1)^2}$.
\end{remark}
\begin{proof}[Proof of Proposition \ref{prop_beta_i4}]
    By Remark \ref{prop_map_with_dual}, $\beta_H = g \circ \beta_{I_4}$, where $\beta_{I_4}$ is a Belyi function for $I_4$, which we want to compute.
    \begin{gather*}
	    \beta_H = \frac{4\beta}{(\beta+1)^2}\quad\Rightarrow\quad \beta_{1,2}(x) = 2\cdot\frac{1\pm\sqrt{1-\beta_H(x)}}{\beta_H(x)} - 1
    \end{gather*}
    Substituting $\beta_H$ into $\beta_{1,2}$, we obtain
    $$\beta_{1,2}(x) = 1 + \frac{25b^2}{128a^5}\cdot(125b^2 \pm \sqrt{5(256a^5 + 3125b^4)}).$$
    Note that $256a^2 + 3125b^2$ is the discriminant of a polynomial $p(x)$ of the form (\ref{eq_p}). Thus,
    $$\beta_{1,2}(x) = 1 + \frac{25b^2}{128a^5}\cdot(125b^2 \pm \sqrt{5}\prod\limits_{i<j}(x_i - x_j)).$$
    Belyi pairs $(\mathcal{B}, \beta_1)$ and $(\mathcal{B}, \beta_2)$ correspond to the same dessin because there exists an automorphism $\varphi$ of Bring's curve $\mathcal{B}$ such that
    $$\beta_1 = \beta_2\circ\varphi.$$
    For example, you can take $\varphi = (12) \in \text{Aut}(\mathcal{B})$. Thus, both functions $\beta_1, \beta_2$ can be considered as a Belyi function for $I_4$.
\end{proof}
\begin{corollary}
    $I_4$ and $I_4^*$ are isomorphic.
\end{corollary}
\begin{proof}
    Note, both Belyi pairs $(\mathcal{B}, \beta_1)$ and $(\mathcal{B}, \beta_2)$ correspond to $I_4$ (see Proof of Proposition ~\ref{prop_beta_i4}). Suppose $\beta$ is a Belyi function for a dessin $M$, then $\frac{1}{\beta}$ is a Belyi function for $M^*$. We will check that 
    $$\frac{1}{\beta_1} = \beta_2.$$
    Indeed, 
    \begin{multline*}
        \beta_{1,2} = 2\frac{1 \pm \sqrt{1 - \frac{256a^5}{256a^5 + 3125b^4}}}{\frac{256a^5}{256a^5 + 3125b^4}} - 1 = 2\frac{1 \pm \sqrt{1 - \frac{4^4a^5}{4^4a^5 + 5^5b^4}}}{\frac{4^4a^5}{4^4a^5 + 5^5b^4}} - 1,\\
	    \frac{1 \pm \sqrt{1 - \frac{256a^5}{256a^5 + 3125b^4}}}{\frac{256a^5}{256a^5 + 3125b^4}} = \frac{4^4a^5 + 5^5b^4 \pm \sqrt{5^5b^4(4^4a^5+5^5b^4)}}{4^4a^5} = 1 + t \pm \sqrt{t(1+t)},
    \end{multline*}
    where $t = \left(\frac{5}{a}\right)^5\left(\frac{b}{4}\right)^4$. Hence,
    \begin{gather*}
        \beta_1 = 2 + 2t + 2 \sqrt{t(1+t)} - 1 = 1 + 2t + 2 \sqrt{t(1+t)},\\
        \frac{1}{\beta_1} = \frac{1}{1 + 2t + 2 \sqrt{t(1+t)}} = \frac{1 +2t - 2\sqrt{t(1+t)}}{(1+2t)^2 - 4(t + t^2)} = 1 +2t - 2\sqrt{t(1+t)} = \beta_2.
    \end{gather*}
\end{proof}
\subsection{The Automorphism Group of $I_4$}
\begin{prop}
    The automorphism group Aut$(I_4)$ of the icosahedron $I_4$ of genus 4 is the alternating group $A_5$, which acts on Bring's curve $\mathcal{B}$ by even permutations of homogeneous coordinates.   
\end{prop}
\begin{proof}
    Consider the Belyi function for $I_4$ 
    $$\beta_{I_4}(x)= 1 + \frac{25b^2}{128a^5}\cdot(125b^2 
        + \sqrt{5}\prod\limits_{1\leqslant i<j \leqslant 5}(x_i - x_j)).$$
    The polynomials $a$ and $b$ are symmetric. So, the product
    $$\prod\limits_{1\leqslant i<j \leqslant 5}(x_i - x_j)$$
    is invariant under even permutations and changes the sign under odd permutations. 
\end{proof}
\begin{remark}\label{rem_i0_i4}
    Considering the Great Dodecahedron as the immersion of the surface of genus 4 into $\mathbb{R}^3$ (Figure \ref{im_i4_r3}), we can visualize automorphisms of $I_4$ as automorphisms of a regular icosahedron. 
\end{remark}
\subsection{Vertices and Face-centres of $I_4$}
\begin{prop}\cite{Zvonkin}
    \begin{enumerate}
        \item The black vertices of $I_4$ are the points that can be obtained from the point $$(1: e^{\frac{2\pi i}{5}}:e^{\frac{4\pi i}{5}}:e^{\frac{6\pi i}{5}}: e^{\frac{8\pi i}{5}})$$ by even permutations of coordinates. 
        \item The white vertices of $I_4$ are the points that can be obtained from the point $$(0: 1: -1: i: -i)$$ by permutations of coordinates.
        \item The face centers of $I_4$ are the points that can be obtained from the point $$(e^{\frac{2\pi i}{5}}:1:e^{\frac{4\pi i}{5}}:e^{\frac{6\pi i}{5}}: e^{\frac{8\pi i}{5}})$$ by even permutations of coordinates.
    \end{enumerate}
\end{prop}

\section{Diagram of Belyi function for $I_4$ decompositions}
\textit{Notation.} Let $M$ be a dessin, let $X$ be the corresponding Riemann surface, and let $G$ be a subgroup in Aut$(M)$. We identify points $z_1, z_2 \in X$ if there exists $f \in G \subset \text{Aut}(M)$ such that $z_1 = f(z_2)$ and obtain the quotient curve $X/_{G}$ and the quotient dessin $M/_{G}$. 

\textit{Relations between quotients.} We now consider the various relationships we might expect between the quotient dessins of $I_4$. Let us consider rotational symmetries $\mathbb{Z}_n$ of $I_4$. If $\mathbb{Z}_n$ and $\mathbb{Z}_n\rtimes \mathbb{Z}_m$ are subgroups of Aut$(I_4)$ then we have commutativity of the following diagram of quotients:
\[
    \begin{tikzcd}
        I_4 \arrow{r}{n} \arrow[swap]{dr}{\mathbb{Z}_n\rtimes \mathbb{Z}_m} & I_4/_{\mathbb{Z}_n} \arrow{d}{m} \\ & I_4/_{\mathbb{Z}_n\rtimes \mathbb{Z}_m}
    \end{tikzcd}
\]
where the arrows are labelled by the symmetry being quotient. Here $n$ and $m$ stand for $\mathbb{Z}_n$ and $\mathbb{Z}_m$ respectively.

Nontrivial subgroups of $A_5$ up to automorphisms:
\begin{itemize}
    \item $\mathbb{Z}_2$,
    \item $V_4 = \mathbb{Z}_2 \times \mathbb{Z}_2$,
    \item $A_4 = V_4 \rtimes \mathbb{Z}_3$,
    \item $A_3 = \mathbb{Z}_3$,
    \item $S_3 = \mathbb{Z}_3\rtimes \mathbb{Z}_2$,
    \item $\mathbb{Z}_5$,
    \item $D_{10} = \mathbb{Z}_5 \rtimes \mathbb{Z}_2$.
\end{itemize}
\begin{figure}
    \centering
    \includegraphics[scale=0.5]{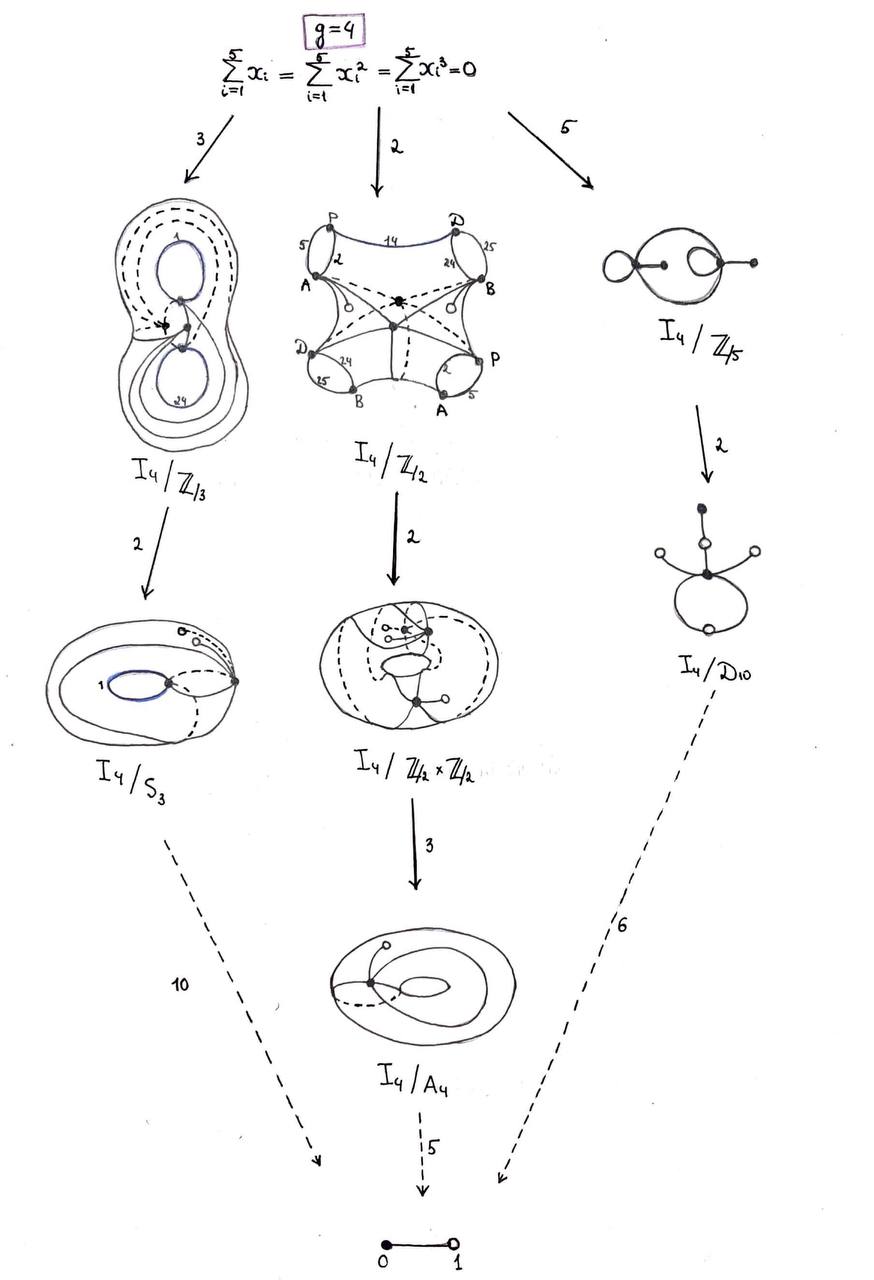}
    \caption{Diagram of Belyi function for $I_4$ decompositions}
    \label{gensch}
\end{figure}
\begin{figure}
    \centering
    \includegraphics[scale=0.7]{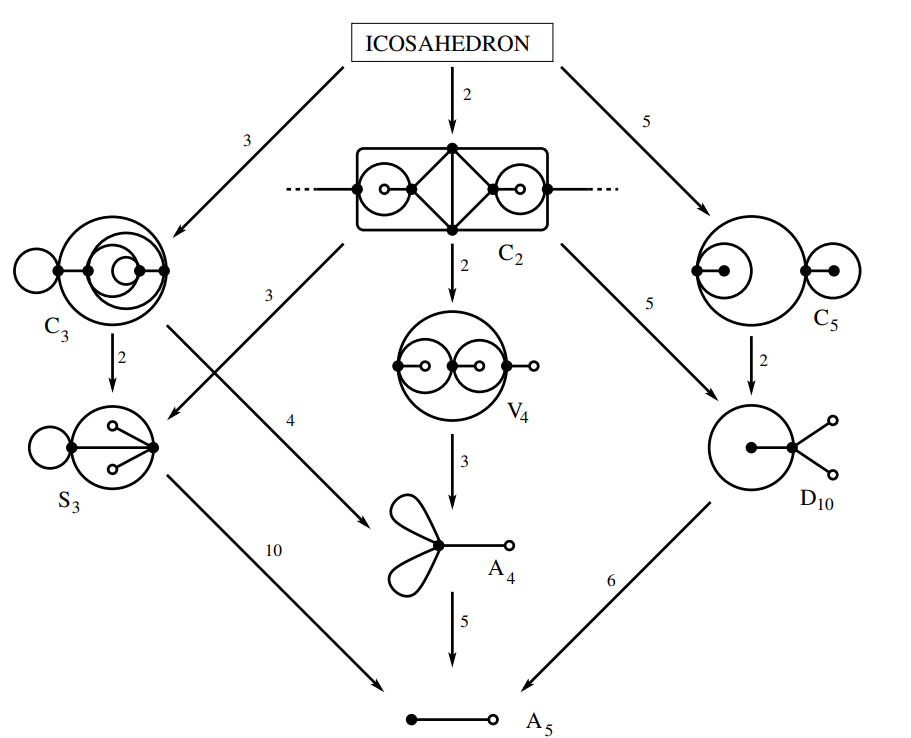}
    \caption{Diagram of Belyi function for a regular icosahedron decompositions}
    \label{im_i0}
\end{figure}
\begin{figure}
    \centering
    \includegraphics[scale=0.5]{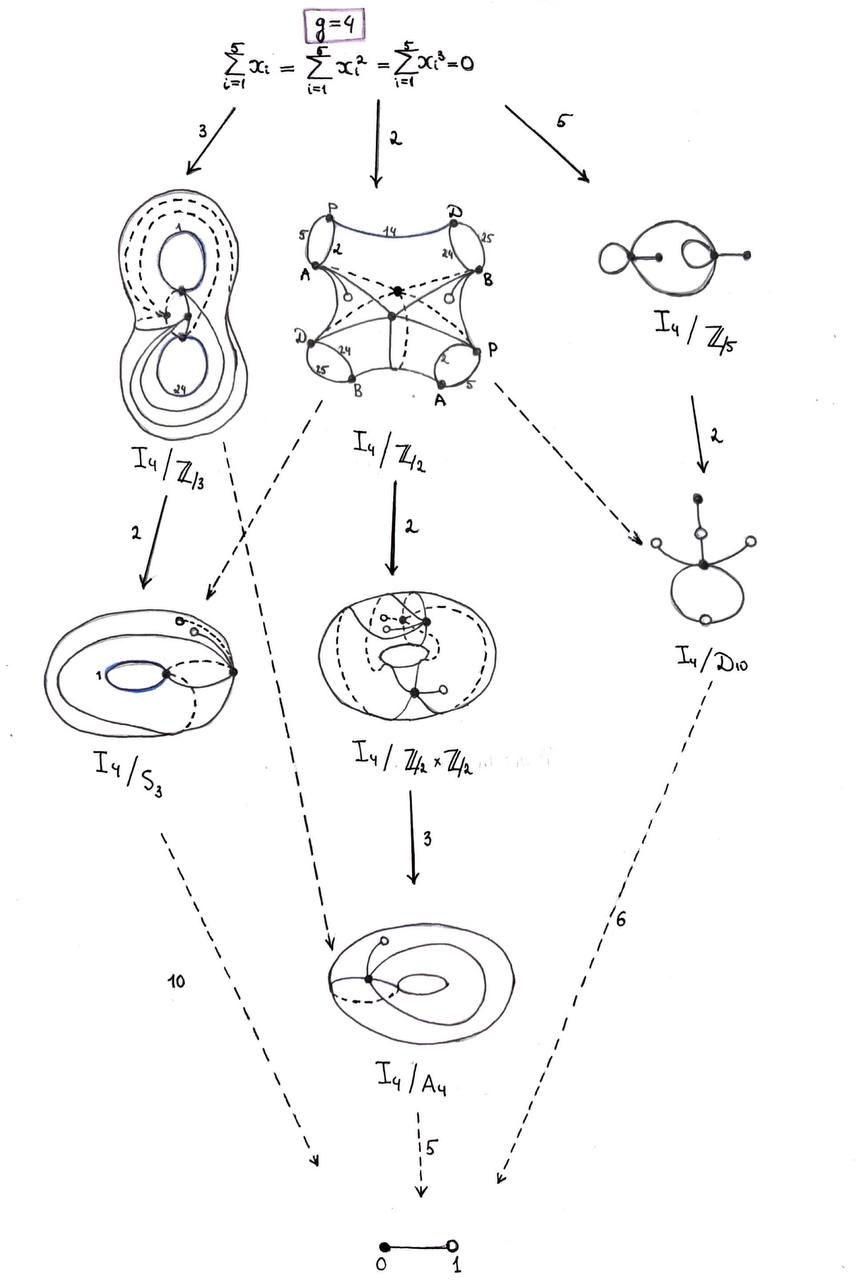}
    \caption{Diagram of Belyi function for $I_4$ decompositions}
    \label{full_gensch}
\end{figure}
In Figure \ref{gensch} you can see all quotient dessins $\{I_4/_{G}\: | \;  G \subset \text{Aut}(I_4)\}.$ There is an arrow from  $I_4/_{G_1}$ to $I_4/_{G_2}$ if $I_4/_{G_2}$ is a quotient of $I_4/_{G_1}$ by a rotational symmetry of some order. An order of а symmetry is placed near the corresponding arrow. In fact, the diagram of a Belyi function for $I_4$ decompositions has more arrows then the diagram shown in Figure \ref{gensch}. If a dessin has no symmetries, it does not mean that the corresponding Belyi function is not decomposable. There exist so called hidden symmetries of dessins \cite{Zvonkin}. 

The diagram of a Belyi function for a regular icosahedron decompositions is obtained by A.K. Zvonkin (see Figure \ref{im_i0}). The main object of this paper is the diagram of a Belyi function for $I_4$ decompositions shown in Figure \ref{full_gensch}. The diagram for $I_4$ is more complicated. In case of a regular icosahedron all interjacent dessins are planar, while in case of $I_4$ dessins are embedded in different algebraic curves.

\section{Belyi pairs for $I_4/_{\mathbb{Z}_5}$ and $I_4/_{D_{10}}$}
\subsection{Belyi function for $I_4/_{\mathbb{Z}_5}$}
\begin{figure}[h!]
    \centering
    \includegraphics[scale=0.5]{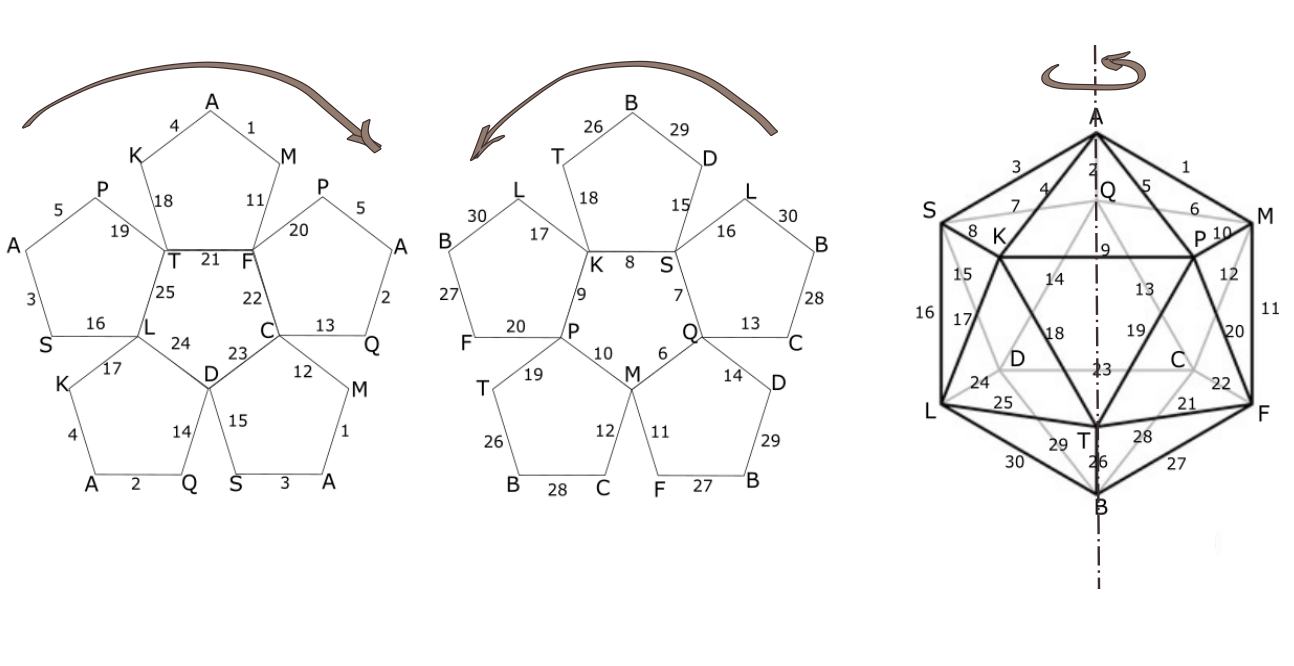}
    \caption{$\mathbb{Z}_5$ acting on the scanning of $I_4$}
    \label{romashka}
\end{figure}
\begin{figure}
    \centering
    \includegraphics[scale=0.5]{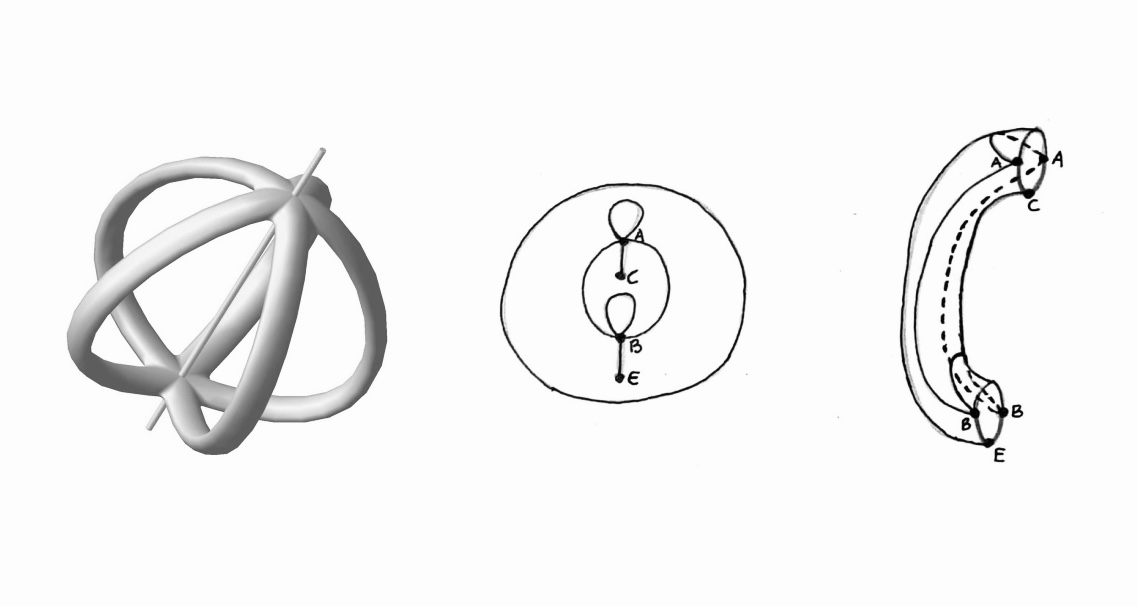}
    \caption{$I_4\rightarrow I_4/_{\mathbb{Z}_5}$}
    \label{im_i4_z5}
\end{figure}
The rotation group $\mathbb{Z}_5$ of a regular icosahedron about the axis shown in Figure \ref{romashka} corresponds to $\mathbb{Z}_5 \subset \text{Aut}(I_4)$ (see Remark \ref{rem_i0_i4}). Now we focus our attention on Figure \ref{im_i4_z5}. On the left you can see a surface of genus 4 which is Bring's curve $\mathcal{B}$. The group of rotations of the surface about the axis is isomorphic to $\mathbb{Z}_5$. The quotient of $I_4$ by the rotational symmetry of order 5 is shown in the middle. On the right there is one part of the surface of genus 4, which we call a leg, and a graph embedded in this leg. If we draw such graphs on each of 5 legs of the surface of genus 4, we will get an embedding of $I_4$ in Bring's curve $\mathcal{B}$.
\begin{prop}\cite{Shabat}
    A Beliy function for plain dessin $I_4/_{\mathbb{Z}_5}$ is equal to
    $$\beta_{I_4/_{\mathbb{Z}_5}}(x) = -\frac{1}{64}\frac{(x-1)^5(x+1)^5(x^2 - 4x -1)}{x^5(x^2+x-1)}.$$
\end{prop}
\subsection{Belyi function for $I_4/_{D_{10}}$}
\begin{figure}[h]
    \centering
    \includegraphics[scale=0.12]{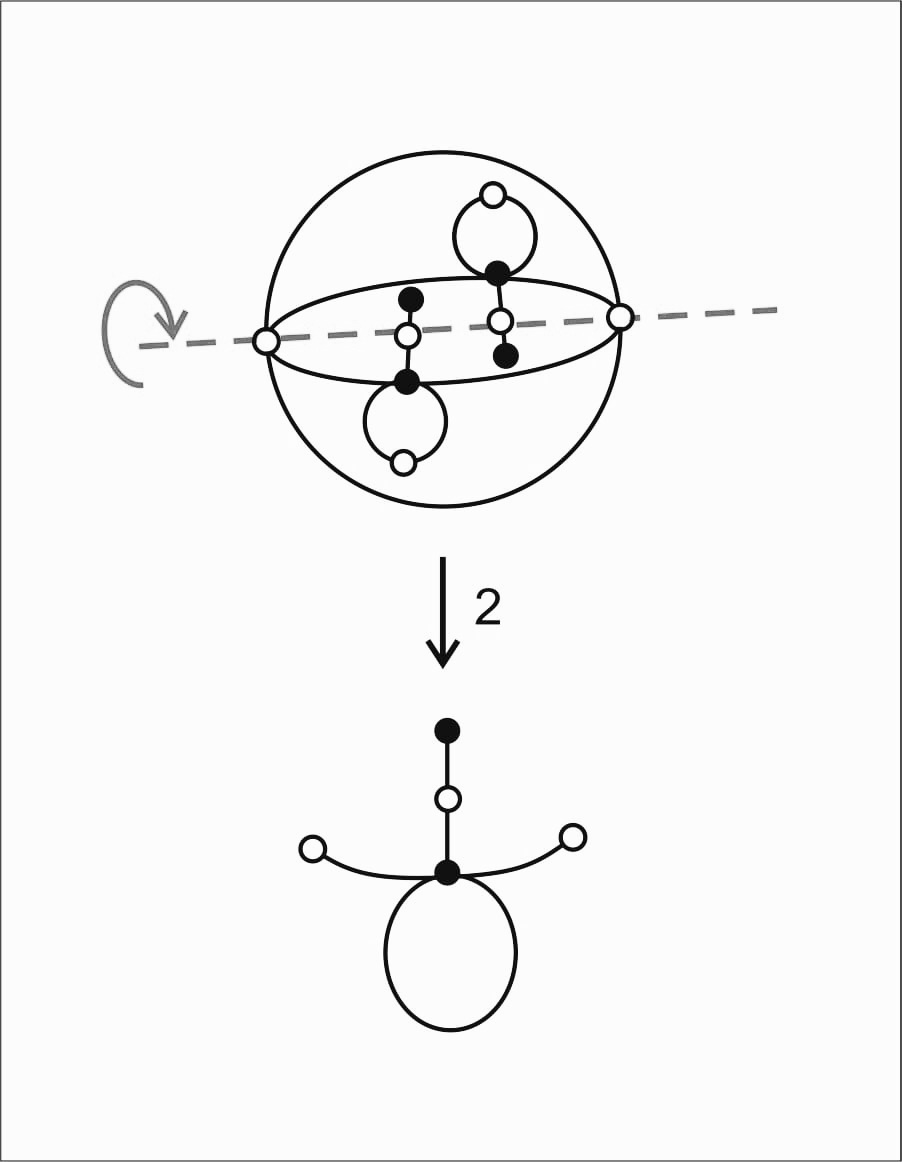}
    \caption{$I_4/_{\mathbb{Z}_5} \rightarrow I_4/_{D_{10}}$}
    \label{im_i4_d10}
\end{figure}
\begin{prop} 
    A Belyi function for $I_4/_{D_{10}}$ (Figure \ref{im_i4_d10}) is equal to
    \begin{equation*}
        \beta_{I_4/_{D_{10}}}(x) = \frac{(x-1)^5((2+i)x+(2-i))}{(x+1)^5((2+i)x -(2-i))}.
    \end{equation*}
\end{prop}
\begin{proof}
    Let $f \in \text{Aut}(I_4/_{\mathbb{Z}_5})$ be an involution such that a quotient of $I_4/_{\mathbb{Z}_5}$ by the action of $f$ is $I_4/_{D_{10}}$. As Figure \ref{im_i4_d10} shows, $f$ maps one black vertex $x = 1$ of degree 5 to another black vertex $x = -1$ of degree 5:
    \begin{gather*}
        f(1) = -1, \quad f(-1) = 1.
    \end{gather*}
    Also, $f$ swaps centers of faces of degree 5:
    \begin{gather*}
        f(0) = \infty, \quad f(\infty) = 0.
    \end{gather*}
    Linear fractional transformation which does the job is
    $$f(x) = -\frac{1}{x} \in \text{Aut}(I_4/_{\mathbb{Z}_5}).$$
    Note that there are two fixed points of $f$, namely $x = \pm i$. In Figure \ref{im_i4_d10} points $x = \pm i$ are shown as white vertices through which the axis of rotation passes.
    
    Consider a dessin $\widetilde{I_4/_{\mathbb{Z}_5}}$ corresponding to the Belyi pair $(\mathbb{CP}^1, \beta_{I_4/_{\mathbb{Z}_5}}\circ g),$ where 
    \begin{gather*}
        x = g(y) = \frac{y-i}{iy-1}\in \text{Aut}(\mathbb{CP}^1). \\
    \end{gather*}
    Thus, the Belyi function for $\widetilde{I_4/_{\mathbb{Z}_5}}$ is
    $$\beta_{I_4/_{\mathbb{Z}_5}}(g(y)) = \frac{(y^2-1)^5((2+i)y^2+(2-i))}{(y^2+1)^5((2+i)y^2 -(2-i))}.$$
    In fact, $\widetilde{I_4/_{\mathbb{Z}_5}}$ is isomorphic to $I_4/_{\mathbb{Z}_5}$. Consider an automorphism $\tilde f$ of $\widetilde{I_4/_{\mathbb{Z}_5}}$
    \begin{gather*}
        \tilde f: y \mapsto -y
    \end{gather*}
    Note, $\tilde f = g^{-1}\circ f\circ g.$ Indeed, 
    \begin{gather*}
        g(\tilde f(x)) = g(-x) = -\frac{1}{g(x)} = f(g(x)).
    \end{gather*}
    A quotient of $\widetilde{I_4/_{\mathbb{Z}_5}}$ by the action of $\tilde f$ is $I_4/_{D_{10}}$. The quotient map is given by $\widetilde{I_4/_{\mathbb{Z}_5}} \rightarrow I_4/_{D_{10}}$ is 
    $$y^2 \mapsto z.$$
    The following diagram is commutative.
    \begin{gather}\label{diag_d10}
        \begin{CD}
            \widetilde{I_4/_{\mathbb{Z}_5}} @>g>> I_4/_{\mathbb{Z}_5}\\
            @Vy^2 \mapsto zVV                             @V\beta_{I_4/_{\mathbb{Z}_5}}VV      \\
             I_4/_{D_{10}}      @>\beta_{I_4/_{D_{10}}}>>         \mathbb{CP}^1
        \end{CD}
    \end{gather}
    So, the Belyi function for $I_4/_{D_{10}}$ is 
    \begin{equation*}
        \beta_{I_4/_{D_{10}}}(z) = \frac{(z-1)^5((2+i)z+(2-i))}{(z+1)^5((2+i)z -(2-i))}.
    \end{equation*}
\end{proof}

\section{Algebraic curves for $I_4/_{\mathbb{Z}_3}$ and $I_4/_{S_3}$}
\subsection{Hyperelliptic curve for $I_4/_{\mathbb{Z}_3}$}
\begin{figure}
    \centering
    \includegraphics[scale=0.35]{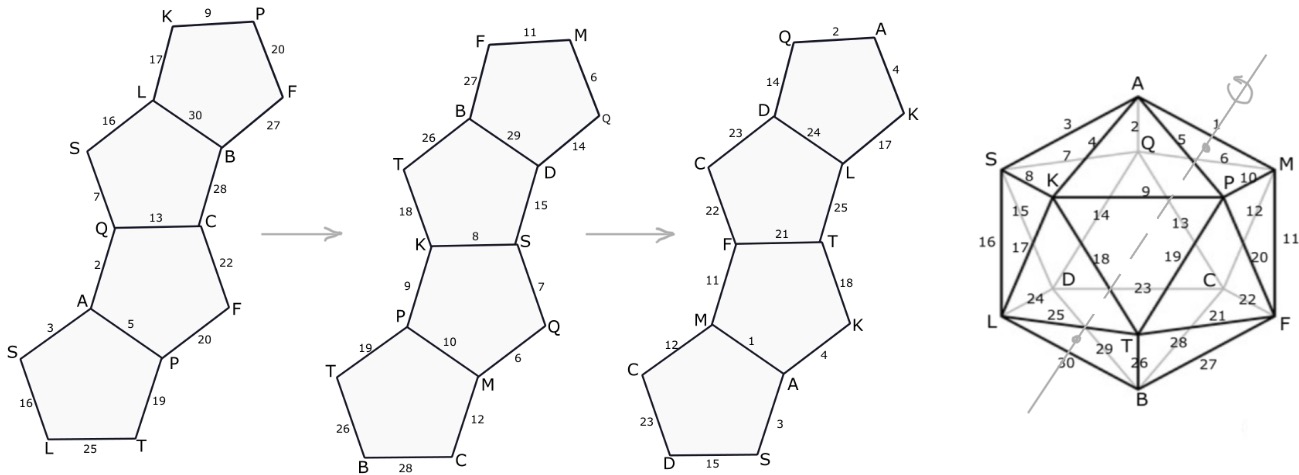}
    \caption{$\mathbb{Z}_3$ acting on the scanning of $I_4$}
    \label{im_rotation_z3}
\end{figure}
The rotation group $\mathbb{Z}_3$ of the icosahedron about the axis as shown on the right in Figure $\ref{im_rotation_z3}$ corresponds to $\mathbb{Z}_3 \subset \text{Aut}(I_4)$ (see Remark \ref{rem_i0_i4}). Also, we can visualize automormphisms of $I_4$ as automorphisms of the scanning of $I_4$ as it is shown on the left in Figure $\ref{im_rotation_z3}$. Now we focus our attention on Figure \ref{im_i4_z3}. On the left you can see a surface of genus 4 which is Bring's curve $\mathcal{B}$. The group of rotations of the surface about the axis is isomorphic to $\mathbb{Z}_3$. The quotient of $I_4$ by the rotational symmetry of order 3 is shown in the middle. On the right there is one part of the surface of genus 4, which we call a leg, and a graph embedded in this leg. If we draw such graphs on each of 3 legs of the surface of genus 4, we will get an embedding of $I_4$ in Bring's curve $\mathcal{B}$.
\begin{figure}
    \centering
    \includegraphics[scale=0.45]{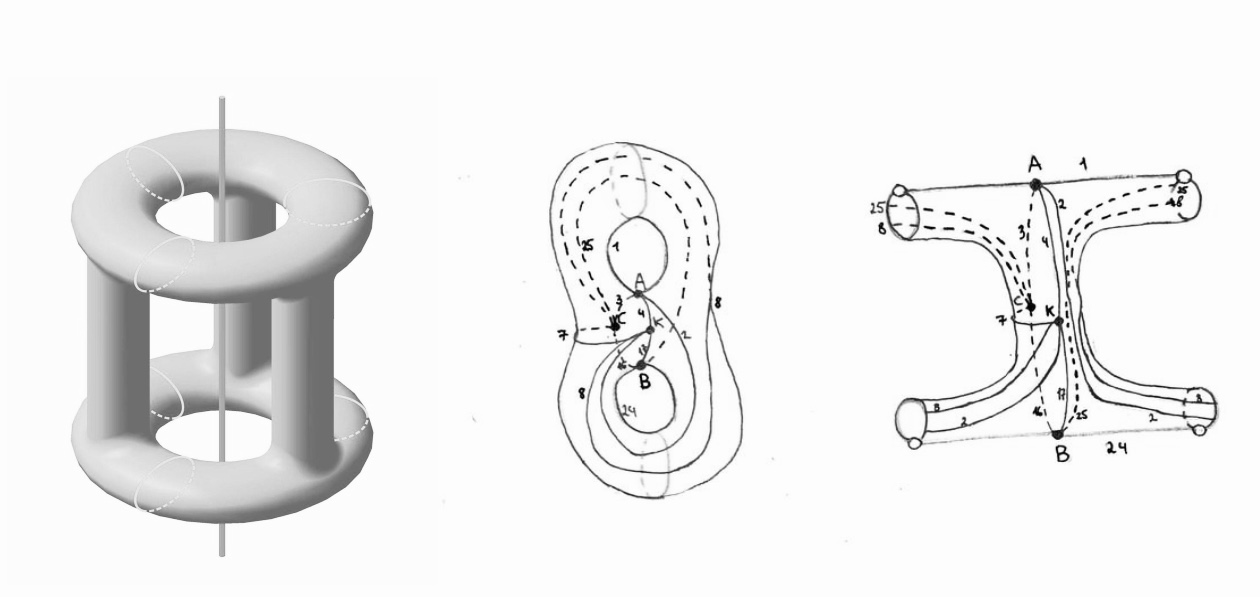}
    \caption{$I_4 \rightarrow I_4/_{\mathbb{Z}_3}$}
    \label{im_i4_z3}
\end{figure}
\begin{prop}\label{prop_z3}
    $\mathcal{B}/_{\mathbb{Z}_3}$ is biholomorphic to the algebraic curve defined by the affine equation
    \begin{equation}\label{eq_z3}
        y^2 = 432x^6+648x^5+945x^4+1350x^3+945x^2+648x +432.
    \end{equation}
\end{prop}
\begin{proof}
    Since the genus of $\mathcal{B}/_{\mathbb{Z}_3}$ is equal to $2$, it follows that $\mathcal{B}/_{\mathbb{Z}_3}$ is a hyperelliptic curve. Hence, $\mathcal{B}/_{\mathbb{Z}_3}$ is given by an equation of the form 
    $$y^2 = P(x),$$
    where roots of a polynomial $P(x)$ are ramification points of a hyperelliptic covering. 

    Consider the projection
    \begin{gather*}
        \pi: \mathcal{B} \rightarrow \mathbb{CP}^1, \quad \pi: (x_1:x_2:x_3:x_4:x_5) \mapsto (x_4:x_5).
    \end{gather*}
    A regular point $(x_4:x_5) \in \mathbb{CP}^1$ has $6$ preimeges on Bring's curve:
    \begin{gather*}
        (x_1:x_2:x_3:x_4:x_5),\:
        (x_2:x_3:x_1:x_4:x_5),\:
        (x_3:x_1:x_2:x_4:x_5),\\
        (x_2:x_1:x_3:x_4:x_5),\:
        (x_1:x_3:x_2:x_4:x_5),\:
        (x_3:x_2:x_1:x_4:x_5).
    \end{gather*}
    These six points on Bring's curve are mapped into two points on $\mathcal{B}/_{\mathbb{Z}_3}$, where $\mathbb{Z}_3 = \langle(123)\rangle \subset \text{Aut}(I_4).$
    \begin{gather*}
        [(x_1:x_2:x_3:x_4:x_5)]_{\langle(123)\rangle} = [(x_2:x_3:x_1:x_4:x_5)]_{\langle(123)\rangle} = [(x_3:x_1:x_2:x_4:x_5)]_{\langle(123)\rangle},\\
        [(x_2:x_1:x_3:x_4:x_5)]_{\langle(123)\rangle} = [(x_1:x_3:x_2:x_4:x_5)]_{\langle(123)\rangle} = [(x_3:x_2:x_1:x_4:x_5)]_{\langle(123)\rangle}.
    \end{gather*}
    We see that $\pi$ induces a two-to-one covering of $\mathbb{CP}^1$ by $\mathcal{B}/_{\mathbb{Z}_3}$, i.e. a hyperelliptic covering. Note that the hyperelliptic covering and $\pi$ have the same ramification points.

    A point $(x_4:x_5)$ is a ramification point of $\pi$ iff the system (\ref{eqbring}) has a solution $(x_1, x_2, x_3)$ in terms of $x_4, x_5$ such that $x_1 = x_2$. Setting $x_1 = x_2$ in (\ref{eqbring}), we obtain the following system:
    \begin{equation}\label{eq1=2}
        \begin{cases}
            2x_1 + x_3 + x_4 + x_5 = 0,\\
            2x_1^2 + x_3^2 + x_4^2 + x_5^2 = 0,\\
            2x_1^3 + x_3^3 + x_4^3 + x_5^3 = 0.
        \end{cases}
   \end{equation}
  We need to solve the system (\ref{eq1=2}) for $\frac{x_4}{x_5}$.
  The first equation in the system (\ref{eq1=2}) gives 
  $$x_3 = -2x_1-x_4-x_5.$$
  Now substitute this expression for $x_3$ into the second and the third equations of (\ref{eq1=2}):
  \begin{equation}\label{eq_proof_z3_1}
      \begin{cases}
          6x_1^2 + 4x_1(x_4 + x_5) + x_4^2 + x_5^2 + (x_4 + x_5)^2 = 0,\\
          2x_1^3 - (x_4 + x_5 + 2x_1)^3 + x_4^3 + x_5^3 = 0.
      \end{cases}
  \end{equation}
  We set
  $$t = \frac{x_1}{x_5}, \quad x = \frac{x_4}{x_5}.$$
  The system (\ref{eq_proof_z3_1}) in variables $t, x$ is
  \begin{equation}\label{eq_proof_z3_2}
      \begin{cases}
          6t^2 + 4tx + 4t + 2x^2 + 2x + 2 = 0,\\
          6t^3 + 12t^2x + 12t^2 + 6tx^2 + 12tx + 6t + 3x^2 + 3x =0.
      \end{cases}
  \end{equation}
  We multiply the first equation in the system (\ref{eq_proof_z3_2}) by $(\frac{1}{2} - 3t)$ and sum it with the second equation. The system (\ref{eq_proof_z3_2}) is equivalent to
  \begin{equation}\label{eq1=2sq}
      \begin{cases}
          3t^2 + 2t(x+1) + x^2 + x + 1 = 0,\\
          4t^3 + 3t^2 + 2t + 1 = 0.
      \end{cases}
  \end{equation}
  The first equation in the system (\ref{eq1=2sq}) is quadratic in terms of $t$, the roots of this equation are 
  $$t_{1,2} = \frac{-x-1\pm\sqrt{(x+1)^2-3(x^2 + x + 1)}}{3}.$$
  Substituting the expressions $t_{1,2}$ into the second equation of the system (\ref{eq1=2sq}), we obtain two equations:
  \begin{equation}\label{eq_proof_z3_5}
      20x^3 + 15x^2 + 15x + 20 + (4x^2 + 2x + 4)\sqrt{-2x^2 - x - 2} = 0,
  \end{equation}
  \begin{equation}\label{eq_proof_z3_6}
      20x^3 + 15x^2 + 15x + 20 - (4x^2 + 2x + 4)\sqrt{-2x^2 - x - 2} = 0.
  \end{equation}
  We conclude that
  \begin{multline*}
      P(x) = (20x^3 + 15x^2 + 15x + 20 + (4x^2 + 2x + 4)\sqrt{-2x^2 - x - 2})(20x^3 + 15x^2 + 15x + 20 - (4x^2 + 2x + 4)\sqrt{-2x^2 - x - 2}) = \\
      = 432x^6+648x^5+945x^4+1350x^3+945x^2+648x +432.
  \end{multline*}
\end{proof}

\subsection{Elliptic curve for $I_4/_{S_3}$}
\begin{prop}\label{prop_s3}
    $\mathcal{B}/_{S_3}$ is given by the following equation:
    \begin{equation}
        y^2 = x^4 + 80x^3 + 125x^2 + 50x
    \end{equation}
\end{prop}
\begin{figure}
    \centering
    \includegraphics[scale=0.25]{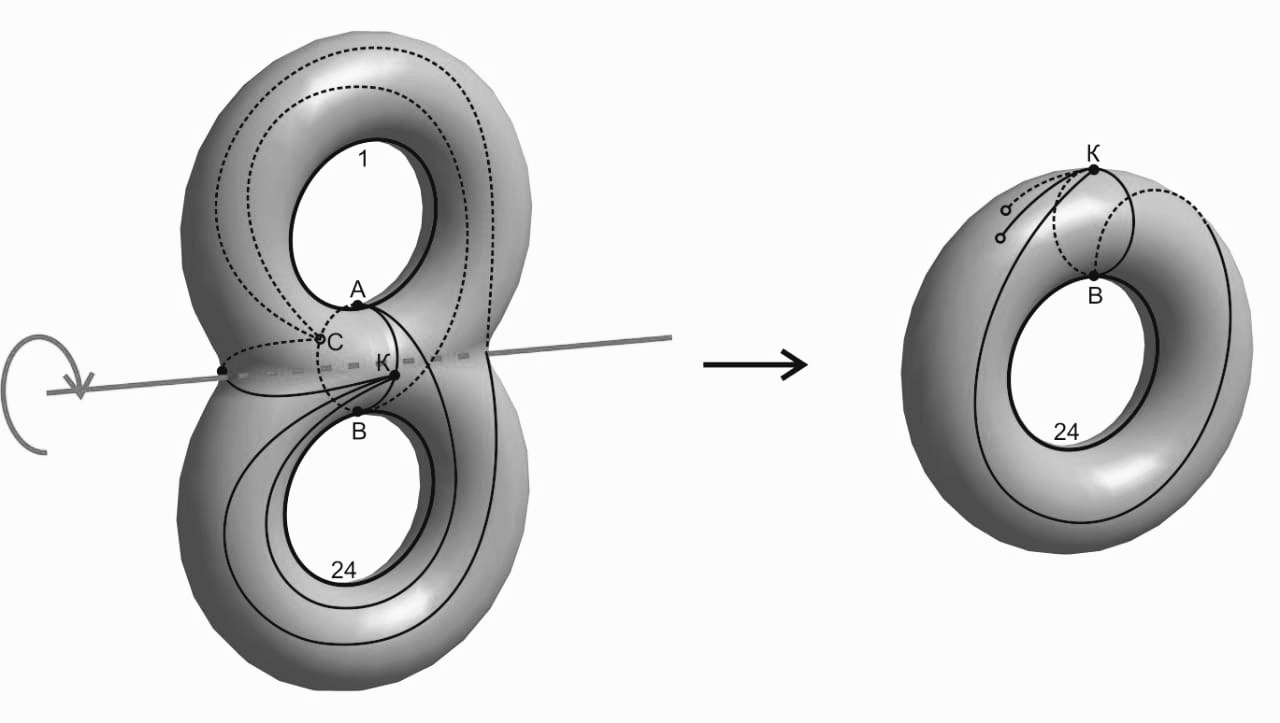}
    \caption{$I_4/_{\mathbb{Z}_3} \rightarrow I_4/_{S_3}$}
    \label{im_i4_s3}
\end{figure}
\textit{Notation.} By $(12)(45)^{\mathbb{Z}_3}$ we denote the automorphism of the curve $\mathcal{B}/_{\mathbb{Z}_3}$ such that the following diagram is commutative: 
\begin{gather}
    \begin{CD}
        \mathcal{B} @>(12)(45)>> \mathcal{B}\\
        @VVV                      @VVV      \\
        \mathcal{B}/_{\langle(123)\rangle} @>(12)(45)^{\mathbb{Z}_3}>> \mathcal{B}/_{\langle(123)\rangle}
    \end{CD}
\end{gather}
\textit{Sketch of proof of Proposition \ref{prop_s3}.} The curve $\mathcal{B}/_{\langle(123)\rangle}$ is given by the equation (\ref{eq_z3}). We compute an equation of the quotient of $\mathcal{B}/_{\langle(123)\rangle}$ by the action of $(12)(45)^{\mathbb{Z}_3} \in \text{Aut}(\mathcal{B}/_{\mathbb{Z}_3})$ in order to find $\mathcal{B}/_{S_3},$ where
$$S_3 = \langle (123), (12)(45) \rangle \subset \text{Aut}(I_4).$$
\begin{prop}
    The automorphism $(12)(45)^{\mathbb{Z}_3} \in \text{Aut}(\mathcal{B}/_{\mathbb{Z}_3})$ acts on the curve $\mathcal{B}/_{\mathbb{Z}_3}$, given by the equation (\ref{eq_z3}), as follows:
    $$(12)(45)^{\mathbb{Z}_3}: (x, y) \mapsto (\frac{1}{x}, -\frac{y}{x^3}).$$
\end{prop}
Figure \ref{im_i4_s3} visualizes the action of $(12)(45)^{\mathbb{Z}_3}$ on $\mathcal{B}/_{\mathbb{Z}_3}$.
\begin{proof}
    In the proof of Proposition \ref{prop_z3} we use the hyperelliptic covering 
    $$\mathcal{B}/_{\mathbb{Z}_3} \rightarrow \mathbb{CP}^1,$$
    induced by the projection
    \begin{equation*}
       \pi: \mathcal{B} \rightarrow \mathbb{CP}^1, \quad \pi:(x_1:x_2:x_3:x_4:x_5) \mapsto (x_4:x_5),
    \end{equation*}
    to find the equation (\ref{eq_z3}):
     $$y^2 = 432x^6+648x^5+945x^4+1350x^3+945x^2+648x +432$$
    of the quotient of Bring's curve $\mathcal{B}/_{\mathbb{Z}_3}$.
    
    Because the automorphism $(12)(45) \in \text{Aut}(\mathcal{B})$ swaps $x_4$ and $x_5$, the involution $(12)(45)^{\mathbb{Z}_3}\in \text{Aut}(\mathcal{B}/_{\mathbb{Z}_3})$ maps $x = \frac{x_4}{x_5}$ to $\frac{1}{x} = \frac{x_5}{x_4}$.
    Thus, $(12)(45)^{\mathbb{Z}_3}(x, y)$ is either
    $(\frac{1}{x}, \frac{y}{x^3})$ or $(\frac{1}{x}, -\frac{y}{x^3})$.\\
    Assume that 
    $$(12)(45)^{\mathbb{Z}_3}: (x,y) \mapsto (\frac{1}{x}, \frac{y}{x^3}).$$
    Then 
     $$(12)(45)^{\mathbb{Z}_3}(1,y) = (\frac{1}{1}, \frac{y}{1}) = (1, y),$$
    i.e. points $(1,y)$ are fixed. Note, $x =\frac{x_4}{x_5} = 1$ iff $x_4 = x_5$. As you can check, if $x_4 = x_5$ then 
    \begin{gather}\label{eq_prop_aut_z3_1}
        \beta_{I_4}(x_1:x_2:x_3:x_4:x_4) = -1
    \end{gather}
    Let us denote Belyi functions for $I_4/_{\mathbb{Z}_3}$ and $I_4/_{S_3}$ by $\beta_{I_4/_{\mathbb{Z}_3}}$ and $\beta_{I_4/_{S_3}}$ respectively. According to the diagram in Figure \ref{gensch}, the function $\beta_{I_4/_{\mathbb{Z}_3}}$ can be decomposed:
    \begin{equation}\label{prop_aut_z3_decomp}
        \beta_{I_4/_{\mathbb{Z}_3}} = \beta_{I_4/_{S_3}}\circ \varphi, 
    \end{equation}
    where $\varphi: I_4/_{\mathbb{Z}_3} \rightarrow I_4/_{S_3}$ is the quotienting by the action of $(12)(45)^{\mathbb{Z}_3}$. Ramification points of $\varphi$ are fixed points of the involution $(12)(45)^{\mathbb{Z}_3}$. By decomposition (\ref{prop_aut_z3_decomp}), these fixed points $(1,y) \in \mathcal{B}/_{\mathbb{Z}_3}$ are ramification points of $\beta_{I_4/_{\mathbb{Z}_3}}$.This is a contradiction with (\ref{eq_prop_aut_z3_1}). Indeed, all critical values of $\beta_{I_4/_{\mathbb{Z}_3}}$ belong to $\{0,1,\infty \}$ and $-1 \not\in \{0,1,\infty \}$. 
    
    Hence, $(12)(45)^{\mathbb{Z}_3}: (x, y) \mapsto (\frac{1}{x}, -\frac{y}{x^3}).$
\end{proof}
\begin{remark}
    The automorphism $(45)^{\mathbb{Z}_3} \in \text{Aut}(\mathcal{B}/_{\mathbb{Z}_3})$ acts on the curve $\mathcal{B}/_{\mathbb{Z}_3}$, given by equation (\ref{eq_z3}), as follows:
    $$(45)^{\mathbb{Z}_3}: (x, y) \mapsto (\frac{1}{x}, \frac{y}{x^3}).$$
\end{remark}
\begin{proof}[Proof of Proposition \ref{prop_s3}]
    Let us find the quotient of the curve $\mathcal{B}/_{\mathbb{Z}_3}$, given by the equation (\ref{eq_z3})
    $$y^2 = 432x^6+648x^5+945x^4+1350x^3+945x^2+648x +432,$$
    by the action of the automorphism 
    $$(12)(45)^{\mathbb{Z}_3}: (x, y) \mapsto (\frac{1}{x}, -\frac{y}{x^3}).$$
    We set
    \begin{gather}\label{prop_s3_z}
        z(x) = \frac{x-1}{x+1}, \quad w = \frac{4y}{3\sqrt{3}(x+1)^3}
    \end{gather}
    We get the equation (\ref{eq_z3}) in terms of $z$ and $w$:
    \begin{equation}\label{eq_z3_zw}
        w^2 = z^6 + 80z^4 + 125z^2 + 50.
    \end{equation}
    Note, $z(\frac{1}{x}) = -z(x)$. Let $\pm c_i$ be roots of the right-hand side of (\ref{eq_z3_zw}). Then we have
    $$w^2 = \prod\limits_{i=1}^{3}(z^2-c_i^2).$$
    The automorphism $(12)(45)^{\mathbb{Z}_3}$ acts in terms of $z$ and $w$ as follows:
    $$(z, w) \mapsto (-z, -w).$$
    Consider the branched covering of a projective line
    \begin{gather*}
        \mathcal{B}/_{\mathbb{Z}_3} \rightarrow \mathbb{CP}^1,\\
        (z, w) \mapsto z^2.
    \end{gather*}
    A regular point $\tilde z = z^2 \in \mathbb{CP}^1$ has 4 preimages on $\mathcal{B}/_{\mathbb{Z}_3}$:
     \begin{gather*}
        (z, w), \quad (-z, -w), \quad (-z, w), \quad (z, -w).
    \end{gather*}
    These four points are mapped into two points on $\mathcal{B}/_{S_3}$:
    \begin{gather*}
        [(z,w)]_{(z,w)\mapsto (-z,-w)}, \quad [(z,-w)]_{(z,w)\mapsto (-z,-w)}.
    \end{gather*}
    So, we obtain a branched covering $\mathcal{B}/_{S_3} \rightarrow \mathbb{CP}^1$ of degree 2. Ramification points of the covering are $\{0, c_1^2, c_2^2, c_3^2\}$. Thus, the curve $\mathcal{B}/_{S_3}$ is given by 
    $$y^2 = x\prod\limits_{i=1}^{3}(x-c_i^2) =x(x^3 + 80x^2 + 125x + 50).$$
\end{proof}

\section{Algebraic curves for $I_4/_{\mathbb{Z}_2}$, $I_4/_{\mathbb{Z}_2\times\mathbb{Z}_2}$ and $I_4/_{A_4}$}
Let us introduce new homogeneous coordinates in $\mathbb{CP}^4$, following to \cite{Riera}:
    \begin{gather*}
        u_1=\frac{x_1}{2}, \quad u_2 = \frac{x_2+x_4}{2}, \quad u_3 = \frac{x_3+x_5}{2}, \quad u_4 = \frac{x_2-x_4}{2}, \quad u_5 = \frac{x_3-x_5}{2}.
    \end{gather*}
Bring's curve $\mathcal{B}$ in new coordinates $(u_1:u_2:u_3:u_4:u_5)$ is given by
    \begin{equation}\label{eq1}
        \begin{cases}
            u_1 = - u_2 - u_3, \\
            3u_2^2 + 3u_3^2 + 4u_2u_3 + u_4^2 + u_5^2 = 0,\\
            u_2^3 + u_3^3 + 4u_2^2u_3 + 4u_2u_3^2 - u_2u_4^2 - u_3u_5^2 = 0.
        \end{cases}
    \end{equation}
In this section we consider the following automorphisms of Bring's curve $\mathcal{B}$:
\begin{gather*}
    (24)(35): (x_1:x_2:x_3: x_4: x_5) \mapsto (x_1:x_4:x_5: x_2: x_3),\\
    (23)(45): (x_1:x_2:x_3: x_4: x_5) \mapsto (x_1:x_3:x_2: x_5: x_4).
\end{gather*}
The automorphism $(24)(35)$ acts as $\phi_1$ in new coordinates:
\begin{gather*}
    \phi_1: (u_1:u_2:u_3:u_4:u_5) \mapsto (u_1:u_2:u_3:-u_4:-u_5),
\end{gather*}
and the automorphism $(23)(45)$ acts as $\phi_2$: 
\begin{gather*}
    \phi_2: (u_1:u_2:u_3:u_4:u_5) \mapsto (u_1:u_3:u_2: u_5:u_4).
\end{gather*}
\subsection{Hyperelliptic curve for $I_4/_{\mathbb{Z}_2}$}
The rotation group $\mathbb{Z}_2$ of the icosahedron about the axis as shown on the right in Figure $\ref{im_i4_z2_rot}$ corresponds to $\mathbb{Z}_2 \subset \text{Aut}(I_4)$ (see Remark \ref{rem_i0_i4}). Also, we can visualize automormphisms of $I_4$ as automorphisms of the scanning of $I_4$ as it is shown on the left in Figure $\ref{im_i4_z2_rot}$.
\begin{figure}
    \centering
    \includegraphics[scale=0.35]{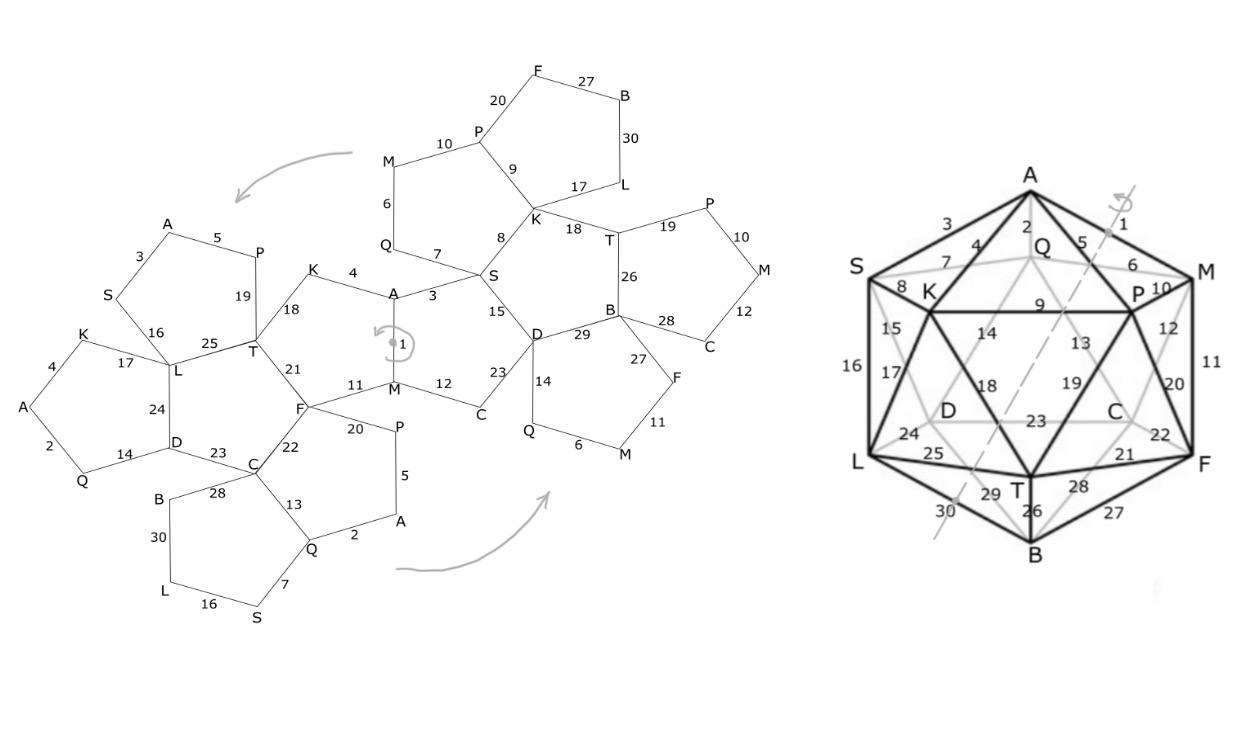}
    \caption{$\mathbb{Z}_2$ acting on the scanning of $I_4$}
    \label{im_i4_z2_rot}
\end{figure}
\begin{figure}
    \begin{minipage}{0.49\linewidth}    
        \center{\includegraphics[scale=0.2]{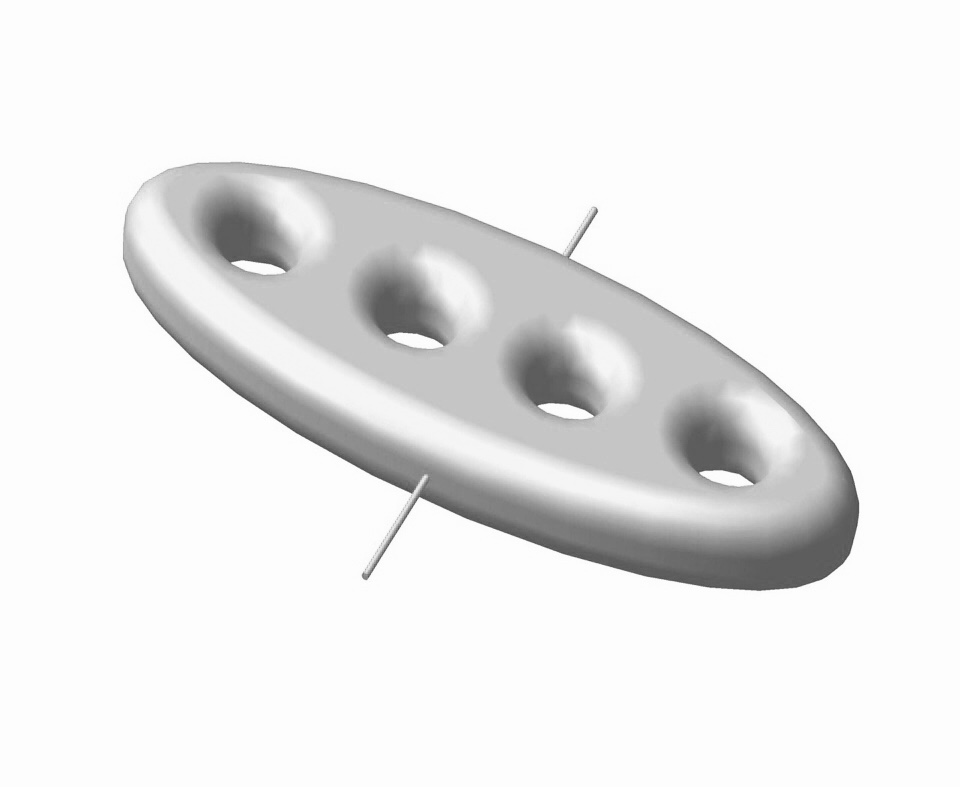}}
        \label{autz2}
    \end{minipage}
    \begin{minipage}{0.49\linewidth}
	\center{\includegraphics[scale=0.45]{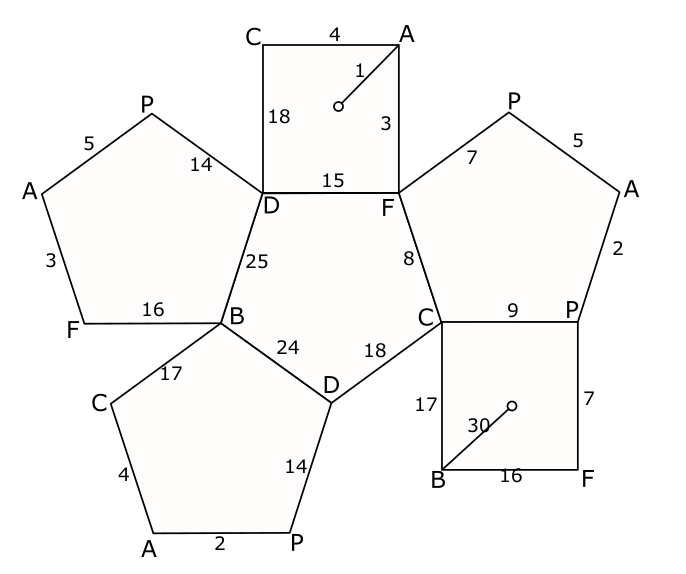}}
        \label{i4z2}
    \end{minipage}
    \caption{$I_4 \rightarrow I_4/_{\mathbb{Z}_2}$}
\end{figure}
\begin{prop}\label{prop_z2}
    $\mathcal{B}/_{\mathbb{Z}_2}$ is biholomorphic to the algebraic curve defined by the affine equation
    \begin{equation}\label{eq_i4_z2}
        y^2 = (4x^3+8x^2+7x+1)(x^3+7x^2+8x+4).
    \end{equation}
\end{prop}
\begin{proof}
    Consider $\mathbb{Z}_2 = \langle \phi_1 \rangle$, where 
    $$\phi_1: (u_1:u_2:u_3:u_4:u_5) \mapsto (u_1:u_2:u_3:-u_4:-u_5).$$
    Since the genus of $\mathcal{B}/_{\mathbb{Z}_2}$ is equal to 2, it follows that $\mathcal{B}/_{\mathbb{Z}_2}$ is a hyperelliptic curve. Hence, $\mathcal{B}/_{\mathbb{Z}_2}$ is given by an equation of the form
    $$y^2 = Q(x),$$
    where roots of a polynomial $Q(x)$ are ramification points of a hyperelliptic covering.

    Consider the projection
    $$\mathcal{B} \rightarrow \mathbb{CP}^1, \quad (u_1:u_2:u_3:u_4:u_5) \mapsto (u_2:u_3).$$
    Let us find the number of preimages of a regular point $(u_2: u_3)$. Note, the second and the third equations from the system (\ref{eq1}) are linear in two variables $u_4^2$ and $u_5^2$:
    \begin{equation}\label{eq_proof_z2_1}
        \begin{cases}
            3u_2^2 + 3u_3^2 + 4u_2u_3 + u_4^2 + u_5^2 = 0,\\
            u_2^3 + u_3^3 + 4u_2^2u_3 + 4u_2u_3^2 - u_2u_4^2 - u_3u_5^2 = 0.
        \end{cases}
    \end{equation}
    If $u_4^2, u_5^2$ are solutions of (\ref{eq_proof_z2_1}) in terms of $u_2, u_3$, then a point $(u_2, u_3) \in \mathbb{CP}^1$ has 4 preimages on Bring's curve $\mathcal{B}$:
    \begin{gather*}
        (u_1:u_2:u_3:u_4:u_5),\: (u_1:u_2:u_3:-u_4:-u_5),\\ (u_1:u_2:u_3:-u_4:u_5),\: (u_1:u_2:u_3:u_4:-u_5).
    \end{gather*}
    These four preimages on Bring's curve are mapped into two points on $\mathcal{B}/_{\mathbb{Z}_2}$:
    \begin{gather*}
        [(u_1:u_2:u_3:-u_4:u_5)]_{(u_4 \leftrightarrow -u_4, u_5 \leftrightarrow -u_5)} = [(u_1:u_2:u_3:u_4: -u_5)]_{(u_4 \leftrightarrow -u_4, u_5 \leftrightarrow -u_5)},\\
         [(u_1:u_2:u_3:u_4:u_5)]_{(u_4 \leftrightarrow -u_4, u_5 \leftrightarrow -u_5)} = [(u_1:u_2:u_3:-u_4:-u_5)]_{(u_4 \leftrightarrow -u_4, u_5 \leftrightarrow -u_5)}.
    \end{gather*}
    Thus, $(u_2:u_3) \in \mathbb{CP}^1$ has 2 preimages on $\mathcal{B}/_{\mathbb{Z}_2}$. So, we have a branched covering of degree 2
    $$\mathcal{B}/_{\mathbb{Z}_2} \rightarrow \mathbb{CP}^1,$$
    i.e. a hyperelliptic covering. 
    
    A point $(u_2: u_3)$ is a ramification point iff the system (\ref{eq_proof_z2_1}) has a solution $(u_4, u_5)$ for given $u_2, u_3$ such that $u_4 = 0$ or $u_5 = 0$. 
    
    Setting $u_4 = 0$ in (\ref{eq_proof_z2_1}), we obtain
    \begin{equation}\label{eq2}
        \begin{cases}
            3u_2^2 + 3u_3^2 + 4u_2u_3 + u_5^2 = 0,\\
            u_2^3 + u_3^3 + 4u_2^2u_3 + 4u_2u_3^2 - u_3u_5^2 = 0.        
        \end{cases}
    \end{equation}
    If we express $u_5^2$ in terms of $u_2, u_3$ from the first equation of (\ref{eq2}) and then substitute the expression in the second equation, we get
    \begin{equation}\label{eq3}
        u_2^3 + 4u_3^3 + 7u_2^2u_3 + 8u_2u_3^2 = 0.
    \end{equation}
    So, if $x = \frac{u_2}{u_3}$ is a root of the equation 
    $$x^3+7x^2+8x+4=0,$$
    then $(u_2:u_3)$ is a ramification point.

    Consider the case when $u_5 = 0$. Putting $u_5 = 0$ in (\ref{eq_proof_z2_1}), we get
    \begin{equation}\label{eq_proof_z2_2}
        \begin{cases}
            3u_2^2 + 3u_3^2 + 4u_2u_3 + u_4^2 = 0,\\
            u_2^3 + u_3^3 + 4u_2^2u_3 + 4u_2u_3^2 - u_2u_4^2 = 0.        
        \end{cases}
    \end{equation}
    If we express $u_4^2$ in terms of $u_2, u_3$ from the first equation of (\ref{eq_proof_z2_2}) and then substitute the expression in the second equation, then we obtain
    \begin{equation}\label{eq4}
        4u_2^3 + u_3^3 + 8u_2^2u_3 + 7u_2u_3^2 = 0.
    \end{equation}
    So, if $x = \frac{u_2}{u_3}$ is a root of the equation
    $$4x^3+8x^2+7x+1=0,$$
    then $(u_2:u_3)$ is a ramification point.

    It follows that  
    $$Q(x) = (4x^3+8x^2+7x+1)(x^3+7x^2+8x+4) = 4x^6 + 36x^5 + 95x^4 + 130x^3 + 95x^2 + 36x + 4.$$
    So, $\mathcal{B}/_{\mathbb{Z}_2}$ is given by 
    $$y^2 = (4x^3+8x^2+7x+1)(x^3+7x^2+8x+4) = 4x^6 + 36x^5 + 95x^4 + 130x^3 + 95x^2 + 36x + 4.$$
\end{proof}
\subsection{Elliptic curve for $I_4/_{\mathbb{Z}_2\times\mathbb{Z}_2}$}
\begin{figure}
    \centering
    \includegraphics[scale=0.35]{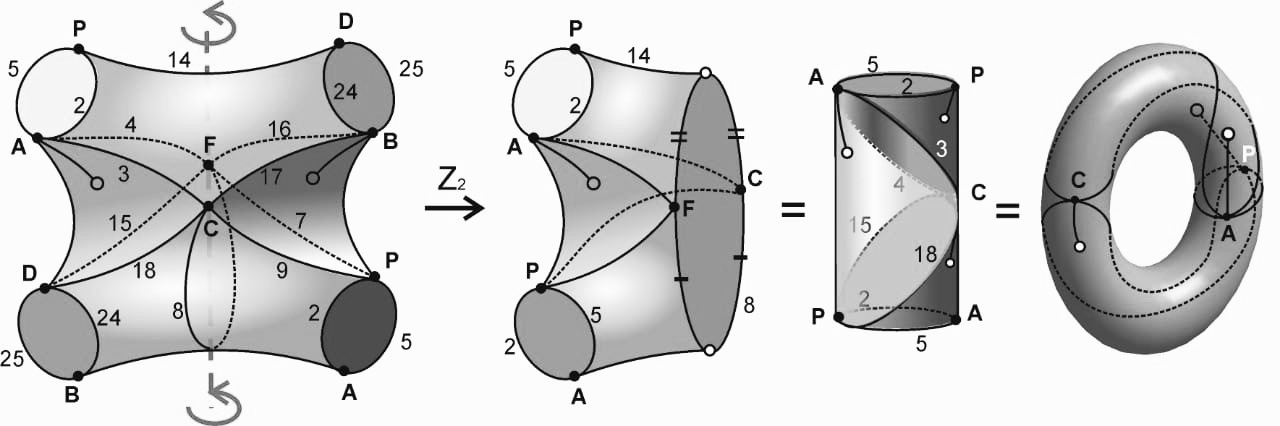}
    \caption{$I_4/_{\mathbb{Z}_2} \rightarrow I_4/_{V_4}$}
    \label{im_i4_v4}
\end{figure}
\begin{prop}\label{prop_z2xz2}
    $\mathcal{B}/_{V_4}$ is given by the following equation:
    \begin{equation}\label{eq_z2xz2}
        y^2 = x^4 + 10x^3 +25x^2 -100x.
    \end{equation}
\end{prop}
\textit{Notation.} By $\phi_2^{\mathbb{Z}_2}$ we denote the automorphism of the curve $\mathcal{B}/_{\mathbb{Z}_2}$ such that the following diagram is commutative:
\begin{gather}\label{diag_aut_z2}
    \begin{CD}
        \mathcal{B} @>\phi_2>> \mathcal{B}\\
        @VVV                      @VVV      \\
        \mathcal{B}/_{\langle\phi_1\rangle} @>\phi_2^{\mathbb{Z}_2}>> \mathcal{B}/_{\langle\phi_1\rangle}
    \end{CD}
\end{gather}
\begin{remark}
    The diagram (\ref{diag_aut_z2}) looks in initial coordinates $(x_1:x_2:x_3:x_4:x_5)$ as follows: 
    \begin{gather}\label{diag_aut_z2_x}
        \begin{CD}
            \mathcal{B} @>(23)(45)>> \mathcal{B}\\
            @VVV                      @VVV      \\
            \mathcal{B}/_{\langle(24)(35)\rangle} @>(23)(45)^{\mathbb{Z}_2}>> \mathcal{B}/_{\langle(24)(35)\rangle}
        \end{CD}
    \end{gather}
\end{remark}
\textit{Sketch of proof of Proposition \ref{prop_z2xz2}.} Consider 
$$V_4 = \{(), (24)(35), (23)(45), (25)(34) \} = \{\ \text{id}, \phi_1, \phi_2, \phi_1 \phi_2\} \subset \text{Aut}(I_4).$$
In order to find $\mathcal{B}/_{V_4}$ we compute an equation of the quotient of the curve $\mathcal{B}/_{\langle\phi_1\rangle}$ (given by (\ref{eq_i4_z2})) by the action of $\phi_2^{\mathbb{Z}_2} \in \text{Aut}(\mathcal{B}/_{\mathbb{Z}_2})$.
\begin{prop}
    The automorphism $\phi_2^{\mathbb{Z}_2} \in \text{Aut}(\mathcal{B}/_{\mathbb{Z}_2})$ acts on the curve $\mathcal{B}/_{\mathbb{Z}_2}$, given by (\ref{eq_i4_z2}), as follows:
    \begin{gather*}
        \phi_2^{\mathbb{Z}_2}: (x,y) \mapsto (\frac{1}{x}, -\frac{y}{x^3}).
    \end{gather*}
\end{prop}
Figure \ref{im_i4_v4} visualizes the action of $\phi_2^{\mathbb{Z}_2}$ on $\mathcal{B}/_{\mathbb{Z}_2}$.
\begin{proof}
    Note that the curve $\mathcal{B}/_{\mathbb{Z}_2}$, given by (\ref{eq_i4_z2})
    $$y^2 = 4x^6 + 36x^5 + 95x^4 + 130x^3 + 95x^2 + 36x + 4,$$
    is a hyperelliptic covering
    $$\mathcal{B}/_{\mathbb{Z}_2} \rightarrow \mathbb{CP}^1,$$
    induced by the projection
    \begin{gather*}
        \mathcal{B} \rightarrow \mathbb{CP}^1, \quad (u_1:u_2:u_3:u_4:u_5) \mapsto (u_2: u_3)
    \end{gather*}
    (see proof of Proposition \ref{prop_z2}). Because the automorhism of Bring's curve 
    $$\phi_2: (u_1:u_2:u_3:u_4:u_5) \mapsto (u_1:u_3:u_2: u_5:u_4)$$
    swaps $u_2$ and $u_3$, the automorphism $\phi_2^{\mathbb{Z}_2}\in \text{Aut}(\mathcal{B}/_{\mathbb{Z}_2})$ maps $x = \frac{u_2}{u_3}$ to $\frac{1}{x} = \frac{u_3}{u_2}$.
    Thus, $\phi_2^{\mathbb{Z}_2} (x,y)$ is either $(\frac{1}{x}, \frac{y}{x^3})$ or $(\frac{1}{x}, -\frac{y}{x^3})$. Assume that 
    $$\phi_2^{\mathbb{Z}_2}: (x,y) \mapsto (\frac{1}{x}, \frac{y}{x^3}).$$
    We will show that there is a contradiction. Setting $x = -1$ in (\ref{eq_i4_z2}), we obtain $y^2 = -4$. Note that the point $(-1, 2i) \in \mathcal{B}/_{\mathbb{Z}_2}$ is not a fixed point of $\phi_2^{\mathbb{Z}_2}\in \text{Aut}(\mathcal{B}/_{\mathbb{Z}_2})$. Indeed,
     $$\phi_2^{\mathbb{Z}_2}(-1, 2i) = (\frac{1}{-1}, \frac{2i}{-1}) = (-1, -2i).$$
    On the other hand, consider the covering $\mathcal{B} \rightarrow \mathbb{CP}^1$. The point $(1:-1)$ has 4 preimages:
    \begin{gather*}
        p_1 = (0:1:-1:i:i), \quad p_2 = (0:1:-1: -i:-i),\\
        q_1 = (0:1:-1:i:-i),\quad  q_2 = (0:1:-1:-i:i),
    \end{gather*}
    which are mapped into 2 points $p =(-1, 2i)$ and  $q =(-1, -2i)$ on $\mathcal{B}/_{\mathbb{Z}_2}$.
    Since 
    \begin{gather*}
        \phi_2(p_1) = p_2, \quad \phi_2(p_2) = p_1,\\
        \phi_2(q_1) = q_1, \quad \phi_2(q_2) = q_2,
    \end{gather*}
    the points $p$ and $q$ are fixed by $\phi_2^{\mathbb{Z}_2}\in \text{Aut}(\mathcal{B}/_{\mathbb{Z}_2})$, i.e. 
    \begin{gather*}
        p = \phi_2^{\mathbb{Z}_2}(p), \quad q = \phi_2^{\mathbb{Z}_2}(q). 
    \end{gather*}
    So, the point $(-1, 2i)$ is a fixed point of $\phi_2^{\mathbb{Z}_2}$. This is a contradiction.

    Hence, $\phi_2^{\mathbb{Z}_2}: (x,y) \mapsto (\frac{1}{x}, -\frac{y}{x^3}).$
\end{proof}
\begin{remark}
   The automorphism $(x,y) \mapsto (\frac{1}{x}, \frac{y}{x^3})$ of the curve $\mathcal{B}/_{\mathbb{Z}_2}$ corresponds to the automorphism $(2345)$ of Bring's curve in initial coordinates.
\end{remark}
\begin{proof}[Proof of Proposition \ref{prop_z2xz2}.]
    Let us find the quotient  of the curve $\mathcal{B}/_{\mathbb{Z}_2}$, given by the equation (\ref{eq_i4_z2}): 
    $$y^2 = 4x^6 + 36x^5 + 95x^4 + 130x^3 + 95x^2 + 36x + 4,$$
    by the action of the automorphism
    \begin{gather*}
        \phi_2^{\mathbb{Z}_2}: (x,y) \mapsto (\frac{1}{x}, -\frac{y}{x^3}).
    \end{gather*}
    We set 
    $$z(x) = \frac{x-1}{x+1},\quad w = \frac{4iy}{(x+1)^3}.$$
    The equation (\ref{eq_i4_z2}) in terms of $z$ and $w$:
    \begin{equation}\label{eq_z2_zw}
        w^2 = z^6 + 10z^4 + 25z^2 - 100.
    \end{equation}
    Note that $z(\frac{1}{x}) = -z(x)$. Let $\pm d_i$ be roots of the right-hand side of (\ref{eq_z2_zw}). Then we have
    $$w^2 = \prod\limits_{i=1}^3(z^2 - d_i^2).$$
    The automorphism $\phi_2^{\mathbb{Z}_2}$ acts in terms of $z$ and $w$ as follows:
    $$(z, w) \mapsto (-z, -w).$$
    Consider the branched covering of a projective line
    \begin{gather*}
        \mathcal{B}/_{\mathbb{Z}_2} \rightarrow \mathbb{CP}^1, \quad (z,w) \mapsto z^2.
    \end{gather*}
    A regular point $\tilde z = z^2 \in \mathbb{CP}^1$ has 4 preimages on $\mathcal{B}/_{\mathbb{Z}_2}$:
    \begin{gather*}
        (z, w), \quad (-z, -w), \quad (-z, w), \quad (z, -w).
    \end{gather*}
     These four points merge into two points on $\mathcal{B}/_{\mathbb{Z}_2 \times \mathbb{Z}_2}$:
    \begin{gather*}
        [(z,w)]_{(z,w)\mapsto (-z,-w)}, \quad [(z,-w)]_{(z,w)\mapsto (-z,-w)}.
    \end{gather*}
    So, we obtain a branched covering $\mathcal{B}/_{\mathbb{Z}_2 \times \mathbb{Z}_2} \rightarrow \mathbb{CP}^1$ of degree 2. Ramification points of the covering are $\{0, d_1^2, d_2^2, d_3^2\}$. Thus, the curve $\mathcal{B}/_{\mathbb{Z}_2 \times \mathbb{Z}_2}$ is given by 
    $$y^2 = x\prod\limits_{i=1}^{3}(x-d_i^2) =x(x^3+10x^2+25x-100).$$
\end{proof}
\begin{remark}
    It turns out that
    $$j_{\mathcal{B}/_{V_4}} = j_{\mathcal{B}/_{S_3}} = -\frac{121945}{32}.$$
    Hence, we have the following curious isomorphism of elliptic curves: 
    $$\mathcal{B}/_{V_4} \simeq \mathcal{B}/_{S_3}.$$
\end{remark}
\subsection{Elliptic curve for $I_4/_{A_4}$}
\begin{figure}
    \centering
    \includegraphics[scale=0.35]{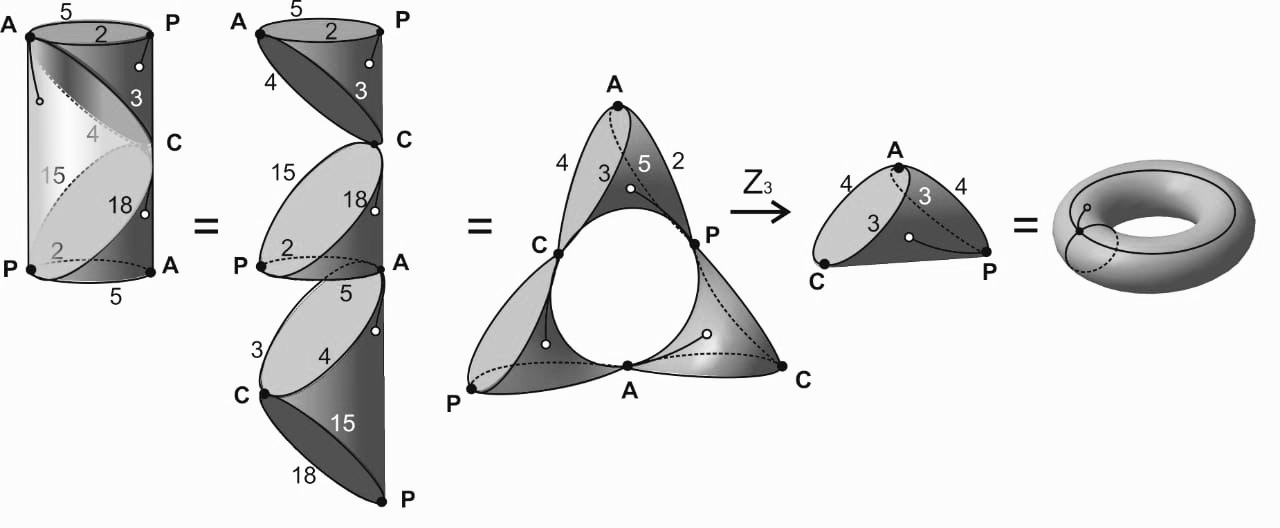}
    \caption{$I_4/_{V_4} \rightarrow I_4/_{A_4}$}
    \label{im_i4_a4}
\end{figure}
\begin{prop}
    $\mathcal{B}/_{A_4}$ is given by the following equation:
    \begin{equation*}
        y^2 + xy + y = x^3 + 549x - 2202.
    \end{equation*}
\end{prop}
\begin{proof}
A curve of genus 1, i.e. an elliptic curve, is a quotient of $\mathbb{C}$ by a lattice $\Lambda$. An elliptic curve $E = \mathbb{C}/_{\Lambda}$ is also an abelian group. The group action by addition
$$E\times E \ni (x,y) \mapsto x+y \in E$$
is a holomorphic map. Being a group, an elliptic curve has a privileged point, the origin $O \in E$.

Consider a point $p \in E$ of order 3 such that $p \neq O$. Translation by $p$ is an automorphism 
\begin{gather*}
    \tau_p: E \rightarrow E, \quad \tau_p: x \mapsto x+p
\end{gather*}
of order 3 of the Riemann surface $E$. It generates a group $\mathbb{Z}_3 = \langle\tau_p\rangle$, acting freely on $E$, and the quotient $E/_{\mathbb{Z}_3}$ is an elliptic curve. 

Figure \ref{im_i4_a4} illustrates that the map $I_4/_{V_4} \rightarrow I_4/_{A_4}$ is an unbranched covering of order 3. Thus, $\mathcal{B}/_{A_4}$ is a quotient of the elliptic curve $\mathcal{B}/_{V_4}$ by the action of $\langle\tau_p\rangle$ where $p \in \mathcal{B}/_{V_4}$ is a non-zero point of order 3. Note that standard Weierstrass form of $\mathcal{B}/_{V_4}$ is
$$y^2 + xy + y = x^3-76x + 298.$$
Searching the LMFDB (\cite{LMFDB}) shows that $(\mathcal{B}/_{V_4})/_{\mathbb{Z}_3}$ is 
$$y^2 + xy + y = x^3 + 549x - 2202.$$
\end{proof}

\section*{Acknowledgments}
We would like to express our deep gratitude to Alexander Zvonkin for turning our attention to such an interesting problem, and to participants of the seminar "Graphs on surfaces and curves over number fields" for their lively interest in the theory of dessins d'enfants. Also, we show appreciation for HSE Faculty of Mathematics for our opportunity of working together. 
\nocite{*}
\printbibliography
N. Ya. Amburg:
\\
NRC "Kurchatov institute", Moscow,  123182,
Russia
\\
Faculty of Mathematics, National Research University Higher School of  
Economics,  Moscow,  119048, Russia
\\
e-mail: amburg@mccme.ru

M. A. Kovaleva:
\\
Faculty of Mathematics, National Research University Higher School of  
Economics,  Moscow,  119048, Russia
\\
e-mail: makovaleva\underline{ }2@edu.hse.ru
\end{document}